\documentclass[a4paper,11pt]{amsart}

\usepackage{hyperref} 
\usepackage[lite]{amsrefs} 
\usepackage{microtype} 
\usepackage{verbatim} 
\usepackage{color}  
\usepackage{amsmath,amssymb} 
\usepackage[all]{xy} 
\usepackage{enumerate} 
\usepackage{mathtools} 
\usepackage{siunitx}
\usepackage{bm}
\usepackage{longtable}
\usepackage{afterpage}

\title{The shrinkage type of knots}
\author{Holger Kammeyer}
\address{Karlsruhe Institute of Technology\\ Institute for Algebra and Geometry\\ Englerstr. 2, Mathebau (20.30)\\ 76131 Karlsruhe\\ Germany}
\email{holger.kammeyer@kit.edu}
\urladdr{www.math.kit.edu/iag7/~kammeyer/}


\newtheorem{theorem}{Theorem}

\newtheorem{proposition}{Proposition}
\newtheorem{lemma}{Lemma}

\theoremstyle{definition}
\newtheorem{definition}{Definition}

\theoremstyle{remark}

\makeatletter
   \let\c@corollary=\c@theorem
   \let\c@proposition=\c@theorem
   \let\c@lemma=\c@theorem
   \let\c@definition=\c@theorem
   \let\c@remark=\c@theorem
   \let\c@example=\c@theorem
   \let\c@equation=\c@theorem
   \let\c@conjecture=\c@theorem
   \let\c@question=\c@theorem
\makeatother

\newcommand*{\MRref}[2]{ \href{http://www.ams.org/mathscinet-getitem?mr=#1}{MR \textbf{#1}}}
\newcommand*{\arXiv}[1]{ \href{http://www.arxiv.org/abs/#1}{arXiv:\textbf{#1}}}

\newcommand*{\Z}{\mathbb Z}
\newcommand*{\Q}{\mathbb Q}
\newcommand*{\R}{\mathbb R}
\newcommand*{\C}{\mathbb C}
\newcommand*{\ima}{\textup{i}}
\newcommand*{\diff}{\textup{d}}
\newcommand*{\muinf}{\underline{\smash{\mu}}}
\newcommand*{\musup}{\overline{\mu}}

\DeclareMathOperator{\rank}{rank}

\DeclareMathOperator{\im}{im}

\DeclareMathOperator{\coker}{coker}

\newcounter{commentcounter}

\usepackage{ifthen,srcltx}
\newcommand{\showcomments}{yes}

\newsavebox{\commentbox}
\newenvironment{com}%
{\ifthenelse{\equal{\showcomments}{yes}}%
{\footnotemark
        \begin{lrbox}{\commentbox}
        \begin{minipage}[t]{1.25in}\raggedright\sffamily\tiny
        \footnotemark[\arabic{footnote}]}
{\begin{lrbox}{\commentbox}}}
{\ifthenelse{\equal{\showcomments}{yes}}
{\end{minipage}\end{lrbox}\marginpar{\usebox{\commentbox}}}
{\end{lrbox}}}

\DefineSimpleKey{bib}{myurl}

\newcommand\myurl[1]{\url{#1}}

\BibSpec{webpage}{%
  +{}{\PrintAuthors} {author}
  +{,}{ \textit} {title}
  +{,}{ } {note}
  +{,}{ \myurl} {myurl}
  +{}{ \parenthesize} {date}
  +{.}{ } {transition}
}

\newcommand{\ignore}[1]{}

\begin{document}

\begin{abstract}  
We study spectral gaps of cellular differentials for finite cyclic coverings of knot complements.  Their asymptotics can be expressed in terms of irrationality exponents associated with ratios of logarithms of algebraic numbers determined by the first two Alexander polynomials.  From this point of view it is natural to subdivide all knots into three different types.  We show that examples of all types abound and discuss what happens for twist and torus knots as well as knots with few crossings.
\end{abstract} 

\maketitle

\section{Introduction}

Let \(K \subset S^3\) be a knot and let \(X = S^3 \setminus \nu K\) be the complement of an open tubular neighborhood of \(K\).  The knot group \(G = \pi_1(X)\) has infinite cyclic abelianization so that for each positive integer \(n\) we have a unique \(n\)-sheeted cyclic covering \(X_n \rightarrow X\).  The asymptotic size of the homology \(H_1(X_n)\) for \(n \rightarrow \infty\) is well understood.  The Betti numbers \(b_1(X_n)\) are bounded while for the torsion subgroup \(\tau H_1(X_n)\) we have
\[ \lim_{n \rightarrow \infty} \frac{\log |\tau H_1(X_n)|}{n} = \frac{1}{2 \pi} \int_0^{2 \pi} \log |\Delta(e^{i\theta})|\, \textup{d}\theta \]
where the right hand side is known as the \emph{logarithmic Mahler measure} \(m(\Delta)\) of the Alexander polynomial \(\Delta(z)\) of \(K\) \cite{Silver-Williams:Mahler}*{Theorem~2.1}.  These results illustrate how a system of finite coverings can determine analytic invariants in the limit:  we obtain \(b_1^{(2)}(X_\infty) = 0\) and \(\rho^{(2)}(X_\infty) = m(\Delta)\) for the first \emph{\(\ell^2\)-Betti number} and the \emph{\(\ell^2\)-torsion} of the infinite cyclic covering \(X_\infty \rightarrow X\).  The first equality is a consequence of \emph{L{\"u}ck approximation}
\cite{Lueck:Approximating}*{Theorem~0.1}
  and the second equality follows from the \emph{approximation conjecture} for \(\ell^2\)-torsion.  In full, the latter is only known for \(\Z\)-coverings \cite{Lueck:ApproxSurvey}*{Remark~6.5} which is precisely the situation at hand.

It is the idea of the so called \emph{Novikov--Shubin numbers}, that for infinite coverings of finite CW complexes yet another homotopy invariant is extractable from the cellular chain complex, beside torsion and homology.  Viewing cellular differentials as bounded operators after \(\ell^2\)-completion, Novikov--Shubin numbers quantify how close these operators are to having a \emph{spectral gap} at zero.  This leads us to studying the asymptotics of spectral gaps of cellular differentials of \(X_n\).  To do so, pick any CW structure of \(X\) and endow \(X_n\) and \(X_\infty\) with the induced structure.  The cell structure determines canonical inner products on the cellular chain modules \(C_k(X_n; \C)\) turning them into finite-dimensional Hilbert spaces.  Therefore the differentials \(d_k^n \colon C_k(X_n; \C) \longrightarrow C_{k-1}(X_n; \C)\) have well-defined \emph{singular values}: the eigenvalues of the positive-semidefinite operator \(\sqrt{{d_k^n}^* d_k^n}\).  The smallest nonzero value amongst them quantifies the size of the spectral gap of \(d^n_k\) around zero.  It is clear that in our context of knot coverings only the second differentials \(d_2^n \colon C_2(X_n; \C) \longrightarrow C_1(X_n; \C)\) are of interest.  So let \(\sigma_n\) be the smallest positive singular value of \(d_2^n\).  Two questions arise: Does \(\sigma_n\) tend to zero for \(n \rightarrow \infty\)? And if yes, how fast?

\begin{definition}
The \emph{lower} and \emph{upper shrinkage rates} of \(K\) are given by
\[ \muinf(K) = \liminf\limits_{n \rightarrow \infty} \frac{-\log \sigma_n}{\log n} \quad \text{ and } \quad \musup(K) = \limsup\limits_{n \rightarrow \infty} \frac{-\log \sigma_n}{\log n}. \]
\end{definition}

Whenever the two rates agree we will simply talk about the \emph{shrinkage rate} \(\mu(K)\). So for example \(\mu(K) = \lambda\) if \(\sigma_n = n^{-\lambda}\).  While the sequence \(\sigma_n\) depends on the chosen CW structure, we prove that the shrinkage rates do not.  In fact, we will show that \(\muinf(K)\) and \(\musup(K)\) are homotopy invariants of \(X\) and thus knot invariants of \(K\) which justifies the notation \(\muinf(K)\) and \(\musup(K)\).  As the main theorem of this article we show how to compute both invariants for all knots in terms of the first and second Alexander polynomials \(\Delta_1\) and \(\Delta_2\) whose definition we recall in Section~\ref{sec:reducedalex}.  To explain the result we need some number theoretic input.  Let
\[ \xi_1 = e^{\pi \ima t_1}, \,\ldots\,,\, \xi_u = e^{\pi \ima t_u} \quad \text{with} \quad t_1, \ldots, t_u \in (0,1) \]
be the distinct roots of the \emph{reduced Alexander polynomial} \(\Lambda = \frac{\Delta_1}{\Delta_2}\) on the upper half of the unit circle with multiplicities \(\muinf_1, \ldots, \muinf_u\).  We recall in Section~\ref{sec:excursion} that each \(t_j\) is either rational or transcendental and each \(t_j\) has a finite \emph{irrationality exponent} \(\nu_j\), a real number which is \(= 1\) if \(t_j\) is rational and \(\ge 2\) otherwise.  We call \(\musup_j = \nu_j \muinf_j\) the \emph{total multiplicity} of the root \(\xi_j\).

\begin{theorem} \label{thm:maintheorem} 
  If \(\Lambda\) has no roots on the unit circle, the sequence \(\sigma_n\) has a positive lower bound and thus \(\mu(K) = 0\).  Otherwise we have
  \[ \muinf(K) = \max_{j = 1, \ldots, u} \{ \muinf_j \} \quad \text{and } \quad \musup(K) = \max_{j = 1, \ldots, u} \{ \musup_j \}. \]
\end{theorem}

In particular, \(\muinf(K)\) is always an integer which is not clear from the definition.  As expected, we confirm in Theorem~\ref{thm:relationtons} that \(\muinf(K)^{-1}\) equals the second Novikov--Shubin number of \(X_\infty\).  Theorem~\ref{thm:maintheorem} says that \(\muinf (K)\) can easily be read off from the Alexander polynomial.  However, to find out the upper shrinkage rate \(\musup(K)\) of a given knot \(K\),  depending on the Alexander polynomials, one runs into a notoriously hard problem of transcendence theory: computing the irrationality exponents of the ratio \(t_i = \frac{\log \xi_j}{\log (-1)}\) of two logarithms of algebraic numbers.  We explain that upper bounds for \(\musup(K)\) result from the theory of linear forms in logarithms.  But these bounds are typically large.  For the three-twist knot \(K = 5_2\), for instance, we obtain \(\muinf(K) = 1\) and \(2 \le \musup(K) \le \num{452131}\).

To see what our result says in practice, it is natural to group knots into three different types according to the qualitative shrinkage behavior.
  
\begin{definition}
  We say that the knot \(K\) has
  \begin{align*} 
     \textit{shrinkage type }
    \begin{cases}
      \textup{I} \\
      \textup{II}  \\
      \textup{III}
    \end{cases}
    \!\!\!\! \textup{if } K \textup{ has shrinkage rate }
    \begin{cases}
      \mu(K) = 0  \\
      \mu(K) > 0 \\
      \muinf(K) < \musup(K).
    \end{cases}
  \end{align*}\end{definition}

It is a curious fact that all three types occur among the two classical infinite families of prime knots.

\begin{theorem} \label{thm:twistandtorusknots}~ 
  \begin{enumerate}[(i)]
  \item Twist knots with even number of half-twists have shrinkage type~I.
  \item Nontrivial torus knots have shrinkage type~II.
  \item Twist knots with odd number \(\ge 3\) of half-twists have shrinkage type~III.
  \end{enumerate}
\end{theorem}

More specifically, we talk about type \(\textup{II}_{\mu}\), \(\textup{III}_{\muinf}\) or \(\textup{III}_{\muinf}^{\musup}\) whenever we want to specify the corresponding parameters.  The observation that type~\(\textup{II}_\mu\) and type~\(\textup{III}_{\muinf}\) can only occur if \(\mu\) and \(\muinf\) are positive integers has a strong converse.

\begin{theorem} \label{thm:knotsofalltypesexist}
For every positive integer \(n\) there is a prime knot of type~\(\textup{II}_n\) and a prime knot of type~\(\textup{III}_n\).
\end{theorem}

Stefan Friedl has informed us that invoking his short note \cite{Friedl:Realization} one can in fact see the stronger statement that all these types already occur among hyperbolic knots.  A list of the shrinkage types of prime knots up to eight crossings is included as Table~\ref{tab:knottypes} on page~\pageref{tab:knottypes}.  Sometimes our inability to compute irrationality exponents makes it impossible to detect the type of a knot.  Amongst the prime knots up to ten crossings this happens in precisely three cases: the knots \(10_{65}\), \(10_{77}\) and \(10_{82}\) as pictured in Figure~\ref{fig:whattype}.
\begin{figure}[htb]
  \includegraphics[width=11cm]{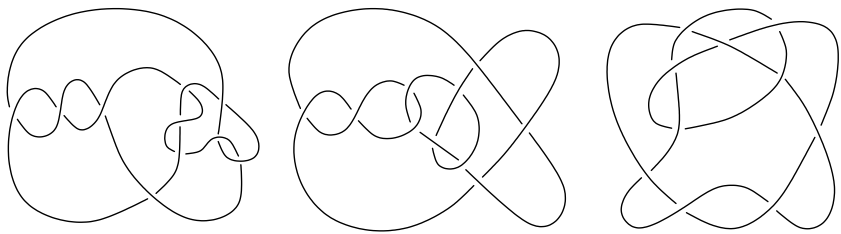}
  \caption{The knots \(10_{65}\), \(10_{77}\) and \(10_{82}\).  Are they type~\(\textup{II}_2\) or \(\textup{III}_2\)?}
  \label{fig:whattype}
\end{figure}
The knots \(10_{65}\) and \(10_{77}\) have \(\Lambda(z) =(z^2-z+1)^2(2z^2-3z+2)\) while \(10_{82}\) has \(\Lambda(z) = (z^2-z+1)^2(z^4-2z^3+z^2-2z+1)\).  Thus deciding if these knots are of type~II or~III is equivalent to deciding whether the transcendental numbers
\[ \frac{\log \left(\frac{3}{4}+\ima \frac{\sqrt{7}}{4}\right)}{\pi \ima} \quad \text{and} \quad \frac{\log \left(\frac{1-\sqrt{2}}{2} +\ima \frac{\sqrt{1+2 \sqrt{2}}}{2}\right)}{\pi \ima} \]
have irrationality exponent larger than two, where \(\log\) is the principal branch.

The outline of this article is as follows.  Section~\ref{sec:reducedalex} has expository character and recalls how reduced Alexander polynomials arise as invariant factors of the second cellular differential of \(X_\infty\).  With a view towards proving Theorem~\ref{thm:knotsofalltypesexist} we include Levine's solution to the realization problem for chains of (reduced) Alexander polynomials.  In Section~\ref{sec:shrinkagerate} we show that the shrinkage rates of the sequence \((\sigma_n)\) are determined by the reduced Alexander polynomials.  Section~\ref{sec:excursion} is an excursion to transcendence theory where we explain what upper bounds on irrationality exponents follow from current results on linear forms in logarithms.  In Section~\ref{sec:proof} we give the proof of Theorem~\ref{thm:maintheorem}.  As one important ingredient we cite a lemma from our earlier work \cite{Kammeyer:Approximating}.  There we had studied a similar finite-dimensional invariant, built out of the minimal singular value and its multiplicity, which more closely mimics the definition of Novikov--Shubin numbers.  In Section~\ref{sec:zoo} we prove the remaining two theorems announced in this introduction.  We give a list containing the shrinkage types of knots up to eight crossings to illustrate that all three types occur frequently.  We conclude the presentation with a discussion of some open questions on the distribution of types and whether these types have a chance to possess any further geometric significance. 

I wish to thank Stefan Friedl, Michel Laurent and Charles Livingston for helpful communication.

\section{Reduced Alexander polynomials}
\label{sec:reducedalex}

We remind the reader that every nonzero matrix \(A \in M(r,s;R)\) over a principal ideal domain \(R\) has a \emph{Smith normal form}.  This means there are invertible matrices \(S \in M(r,r; R)\) and \(T \in M(s,s;R)\) such that the product \(SAT\) has a very special form: there is \(1 \le k \le \min \{r,s\}\) and there are nonzero elements \(a_1, \ldots, a_k \in R\) satisfying \(a_{i+1} \mid a_i\) such that \(SAT\) is bulit from the diagonal \((k \times k)\)-matrix with entries \(a_1, \ldots, a_k \in R\) by adding zero columns and/or zero rows to turn it into an \((r \times s)\)-matrix.  The elements \(a_1, \ldots, a_k\) are unique up to association and are called the \emph{invariant factors} of~\(A\).  A homomorphism \(f \colon M \rightarrow N\) of abstract free, finite rank \(R\)-modules \(M\) and~\(N\) determines a transformation matrix up to multiplication by invertible matrices from left and right so that \(f\) has well-defined invariant factors.

\begin{lemma} \label{lemma:decompinvfactors}
  Let \(R\) be a principal ideal domain and let
  \[ C_2 \xrightarrow{d_2} C_1 \xrightarrow{d_1} C_0 \]
  be homomorphisms of free, finite rank \(R\)-modules such that \(d_1 \circ d_2 = 0\).  Then the homology \(H_1(C_*) = \ker d_1 / \im d_2\) is given by
  \begin{equation} \label{eq:decompinvfactors} H_1(C_*) \cong R^l \oplus R/(a_1) \oplus \cdots \oplus R/(a_k) \end{equation}
  where \(l = \rank \ker d_1 + \rank \ker d_2 - \rank C_2\) and where \(a_1, \ldots, a_k \in R\) are the invariant factors of \(d_2\).
\end{lemma}

\begin{proof}
  Since \(d_1 \circ d_2 = 0\), we have an induced homomorphism
  \[ \overline{d_1} \colon \coker d_2 \rightarrow C_0 \]
  and \(\ker \overline{d_1} = H_1(C_*)\).    Since \(C_0\) is free, the torsion submodule of \(\ker \overline{d_1}\) must be the whole torsion submodule of \(\coker d_2\).  The latter is isomorphic to \(R/(a_1) \oplus \cdots \oplus R/(a_k)\) as can be seen by choosing bases for \(C_2\) and \(C_1\) that turn \(d_2\) into Smith normal form.  The normal form also shows that \(\rank H_1(C_*)\) is given by
  \[ \rank \ker d_1 - \rank \im d_2 = \rank \ker d_1 -(\rank C_2 - \rank \ker d_2) = l. \qedhere \]
\end{proof}

The Smith normal form also easily implies the structure theorem of finitely generated \(R\)-modules which asserts that the decomposition in~\eqref{eq:decompinvfactors} is unique.  Thus the (nonunital) invariant factors of the differentials in a chain complex of free, finite rank \(R\)-modules are determined by the homology.

Let us now return to topology.  We equip the complement \(X = S^3 \setminus \nu K\) of our knot \(K \subset S^3\) with some fixed CW structure and assume that the unique infinite cyclic covering \(X_\infty \rightarrow X\) carries the lifted CW structure.  With this choice the cellular chain complex \((C_*(X_\infty; \Q), d_*)\) with rational coefficients consists of free, finite rank modules over the principal ideal domain of Laurent polynomials \(\Q[z^{\pm 1}]\) where we think of the variable \(z\) as generating the infinite cyclic group of deck transformations.

\begin{definition}
The \emph{reduced Alexander polynomials} \(\Lambda_1(z), \ldots, \Lambda_k(z) \in \Q[z^{\pm 1}]\) of the knot \(K\) are given by the invariant factors of the second cellular differential \(d_2 \colon C_2(X_\infty; \Q) \rightarrow C_1(X_\infty; \Q)\).
\end{definition}

If we agree that \(\Lambda_i(z) = 1\) for \(i > k\), then Lemma~\ref{lemma:decompinvfactors} implies that reduced Alexander polynomials are homotopy invariants of \(X\), thus knot invariants of \(K\).  The polynomials \(\Lambda_i(z)\) are only well-defined up to a factor \(a z^{\pm m}\) with \(a \in \Q^*\).  To obtain the more familiar (unreduced) \emph{Alexander polynomials} \(\Delta_i(z) \in \Z[z^{\pm 1}]\), pick lifts of the 2- and 1-cells in \(X\) to the covering \(X_\infty\).  This defines bases of \(C_2(X_\infty; \Q)\) and \(C_1(X_\infty; \Q)\) which turn \(d_2\) into a matrix with coefficients in the unique factorization domain \(\Z[z^{\pm 1}]\).  The (unreduced) \emph{\(i\)-th Alexander polynomial} \(\Delta_i(z) \in \Z[z^{\pm 1}]\) is the greatest common divisor of the \((k+1-i) \times (k+1-i)\)-minors of this matrix.  It is only defined up to \(\pm z^{\pm m}\).

Gauss's lemma implies that \(\Delta_i(z)\) is also the greatest common divisor of these minors when viewed as elements in \(\Q[z^{\pm 1}]\).  So we must have
\begin{equation} \label{eq:reducedunreduced} \Delta_i(z) = \Lambda_i(z) \cdots \Lambda_k(z) \end{equation}
up to \(a z^{\pm m}\) with \(a \in \Q^*\).  Note that the well-known property \(\Delta_i(1) = \pm 1\) implies that \(\Delta_i(z)\) is a \emph{primitive} polynomial: it has integer coefficients with greatest common divisor one.  It is therefore convenient to normalize each \(\Lambda_i(z)\) to be primitive as well.  For then the product \(\Lambda_i(z) \cdots \Lambda_k(z)\) is likewise primitive and \eqref{eq:reducedunreduced} holds true up to \(\pm z^{\pm m}\).  Finally, let us agree that if we drop the index as in \(\Lambda(z)\) and \(\Delta(z)\), we always refer to \(\Lambda_1(z)\) and \(\Delta_1(z)\).

Which finite sequences of primitive polynomials \(\Lambda_1(z), \ldots, \Lambda_k(z) \in \Z[z^{\pm 1}]\) occur as reduced Alexander polynomials of a knot?  A necessary condition comes from the definition: they need to form a chain of divisors in \(\Q[z^{\pm 1}]\) and thus in \(\Z[z^{\pm 1}]\) by Gauss's lemma.  As we have \(\Lambda_i(z) = \frac{\Delta_i(z)}{\Delta_{i+1}(z)}\), the two familiar properties of Alexander polynomials \(\Delta_i(1) = \pm 1\) and \(\Delta_i(z) = \Delta_i(z^{-1})\) (up to \(\pm z^{\pm m}\)) must also hold for the reduced versions.  Levine  showed that these three necessary conditions together are sufficient for the realization problem.

\begin{theorem}[Levine~\cite{Levine:Characterization}] \label{thm:levine}
  The primitive polynomials \(\Lambda_1(z), \ldots, \Lambda_k(z) \in \Z[z^{\pm 1}]\) occur as reduced Alexander polynomials of a knot if and only if
  \begin{enumerate}[(i)]
  \item \label{item:primitivity} \(\Lambda_i(1) = \pm 1\),
  \item \label{item:symmetry} \(\Lambda_i(z) = \Lambda_i(z^{-1})\) up to \(\pm z^{\pm m}\),
  \item \label{item:divisibility} \(\Lambda_i(z) \mid \Lambda_{i-1}(z)\)
  \end{enumerate}
  for all \(i = 1, \ldots, k\) (with \(\Lambda_0(z) = 0\)).
\end{theorem}

Gordon \cite{Gordon:SomeAspects} provides an alternative to Levine's proof cited above: one can take a closer look on Seifert's construction \cite{Seifert:Geschlecht}*{p.\,588-589} of a knot \(K^\Delta\) with arbitrarily prescribed Alexander polynomial \(\Delta(z) \in \Z[z^{\pm 1}]\) satisfying \(\Delta(1) = 1\) and \(\Delta(z) = \Delta(z^{-1})\).  It turns out that the Alexander module of Seifert's \(K^\Delta\) is actually cyclic: \(H_1(X_\infty; \Z) \cong \Z[z^{\pm 1}]/(\Delta(z))\).  Since the Alexander module of a knot sum is the direct sum of Alexander modules, the knot \(K^{\Lambda_1} \# \cdots \# K^{\Lambda_k}\) does the trick.

\section{Shrinkage rates and reduced Alexander polynomials}
\label{sec:shrinkagerate}

\noindent In this section we show that the shrinkage rates depend only on the reduced Alexander polynomials.  We start by taking a closer look on the definition of \(\sigma_n\).  Recall that the \emph{singular values} of a matrix \(M \in M(r,s;\C)\) are the square roots of the eigenvalues of the positive-semidefinite matrix \(M^* M\).  Clearly the singular values of \(M\) remain unaffected when \(M\) is multiplied by a unitary matrix from left or right.  For a finite CW complex Y we have isomorphisms \(C_p(Y; \C) \cong \C^{N_p}\) for the cellular chain modules with complex coefficients where \(N_p\) is the number of p-cells. These isomorphism are unique up to permutation and sign change of the \(\C\)-summands. As those are unitary transformations, the differentials in \(C_*(Y; \C)\) have well-defined singular values.

Let us now return to our knot complement \(X\) which we have equipped with some fixed CW structure.  By lifting skeleta from \(X\) we obtain a finite CW structure on the unique \(n\)-sheeted finite cyclic covering \(X_n\).  As explained, we are interested in the asymptotics of the sequence \((\sigma_n)\) where \(\sigma_n\) denotes the smallest positive singular value of the second cellular differential \(d^n_2 \colon C_2(X_n; \C) \rightarrow C_1(X_n; \C)\). 

\begin{definition}
  Two sequences \((a_n)\) and \((b_n)\) of positive real numbers are called \emph{dilatationally equivalent} if there is \(C \ge 1\) such that for all \(n\) we have
  \[ C^{-1} a_n \le b_n \le C a_n. \]
\end{definition}

It is apparent that dilationally equivalent sequences \((a_n)\) and \((b_n)\) have the same upper and lower shrinkage rate.  Moreover, \((a_n)\) is bounded from below if and only if \((b_n)\) is. 

\begin{proposition} \label{prop:dependencyonreducedalex}
The dilatation class of \((\sigma_n)_{n > 1}\) depends on the reduced Alexander polynomials of \(K\) only.
\end{proposition}

\begin{proof}
  For each \(2\)-cell and \(1\)-cell in \(X\) we pick a lift to \(X_\infty\).  By projection, this also fixes lifts in each finite covering \(X_n\).  We obtain isomorphisms \(C_2(X_\infty; \C) \cong \C[z^{\pm 1}]^{\,r}\) and \(C_1(X_\infty; \C) \cong \C[z^{\pm 1}]^{\,s}\) under which the differential \(d_2 \colon C_2(X_\infty; \C) \longrightarrow C_1(X_\infty; \C)\) is given by right multiplication with a matrix \(A \in M(r,s; \Z[z^{\pm 1}])\) and \(d^n_2 \colon C_2(X_n; \C) \longrightarrow C_1(X_n; \C)\) is given by right multiplication with the matrix \(A_n \in M(r,s; \Z[z]/(z^n-1))\) obtained from \(A\) by applying the canonical map \(\Z[z^{\pm 1}] \longrightarrow \Z[z]/(z^n-1)\) to the entries.

  Let \(SAT\) be the Smith normal form of \(A\) over \(\Q[z^{\pm1}]\) whose invariant factors are the normalized reduced Alexander polynomials of \(K\).  Clearly we have \((SAT)_n = S_n A_n T_n\) where for \(S\) and \(T\) matrix coefficients are reduced by \(\Q[z^{\pm 1}] \longrightarrow \Q[z]/(z^n-1)\).  Identifying the variable \(z\) with the cyclic, circulant \((n \times n)\)-matrix
  \begin{equation} \label{eq:circulantmatrix}
  \begin{pmatrix}
  0 & 0 & \cdots & 0 & 1 \\
  1 & 0 & & & 0 \\
  0 & 1 & \ddots & & \vdots \\
  \vdots & \ddots & \ddots & 0 & 0 \\
  0 & \cdots & 0 & 1 & 0
  \end{pmatrix}
  \end{equation}
  we can view \(S_n\), \(A_n\) and \(T_n\) as matrices in \(M(nr, nr, \Q)\), \(M(nr, ns, \Z)\) and \(M(ns, ns, \Q)\), repectively.  It is apparent that the smallest positive singular value \(\widetilde{\sigma}_n\) of the matrix \((SAT)_n\) depends on the normalized reduced Alexander polynomials of \(K\) only.  The latter are only unique up to multiplication by \(\pm z^{\pm k}\) but this would correspond to multiplying \((SAT)_n\) by some unitary matrix (from left or right) which leaves \(\widetilde{\sigma}_n\) unchanged.  So the \(\widetilde{\sigma}_n\) are well-defined and the theorem follows once we have shown that \((\widetilde{\sigma}_n)_{n>1}\) and \((\sigma_n)_{n>1}\) are dilatationally equivalent.  Since \(S\) and \(T\) are invertible over \(\Q[z^{\pm 1}]\), right multiplication with \(S S^*\) and \(T T^*\) define invertible positive bounded operators of \((\ell^2 \Z)^r\) and \((\ell^2 \Z)^s\), respectively.  Thus their spectrum is contained in an interval \([C^{-1}, C]\) for some \(C \ge 1\).  Since reducing matrix coefficients is a linear map bounded by one in \(L^1\)-norm, the matrices \(S_n S_n^*\) and \(T_n T_n^*\) have eigenvalues within \([C^{-1}, C]\) as well.  It follows that all singular values of \(S_n\) and \(T_n\) lie in \([C^{-\frac{1}{2}}, C^{\frac{1}{2}}]\).  By standard singular value inequalities as in \cite{Hogben:Handbook}*{24.4.7(c)} (see also \cite{Kammeyer:Approximating}*{Proposition~13}) we obtain \(C^{-1} \widetilde{\sigma}_n \le \sigma_n \le C \widetilde{\sigma}_n\).
\end{proof}

\section{An excursion to transcendence theory} \label{sec:excursion}

How well can a given real number be approximated by rationals?  An attempt to measure this is made by the notion of \emph{irrationality exponent}.  The irrationality exponent \(\nu(t)\) of a real number \(t \in \R\) is the least upper bound of all numbers \(\nu > 0\) for which there are infinitely many pairs of integers \((m,n)\) with \(n > 0\) satisfying
\[ 0 < \left|t - \frac{m}{n}\right| \le \frac{1}{n^{\nu}}. \]
It is easy to see that rational numbers have irrationality exponent one.  \emph{Liouville's theorem} says that an algebraic number of degree \(n \ge 2\) has irrationality exponent at most \(n\).  Thus numbers with infinite irrationality exponent, known as \emph{Liouville numbers}, must be transcendental.  It is easy to write down a Liouville number explicitly, such as \emph{Liouville's constant}
\[ \sum_{n \ge 1} \frac{1}{10^{n!}} = 0.1100010000000000000000010\ldots\, \]
and numbers of this type gave the first explicit decimal examples of numbers known to be transcendental.

There are no numbers with irrationality exponent \(1 < \nu < 2\) in view of \emph{Dirichlet's approximation theorem}:  For all \(t \in \R\) and for any integer \(N > 0\) there exist integers \(m\) and \(n\) with \(1 \le n \le N\) such that \(|nt -m| < \frac{1}{N}\), thus \(|t - \frac{m}{n}| < \frac{1}{n^2}\).  So the irrationality exponent of an irrational algebraic number of degree \(n\) must lie between \(2\) and \(n\).  The \emph{Thue--Siegel--Roth theorem} says that it is actually always~2.  An easy covering argument shows that also Lebesgue-almost all transcendental numbers have irrationality exponent~2.  On the other hand it is not hard to construct a (transcendental) number of any prescribed irrationality exponent \(\nu > 2\).  Y.\,Bugeaud \cite{Bugeaud:Cantor} even showed that one finds such an example within the classical middle third Cantor set.

While these qualitative statements are possible, in general it seems to be a very hard problem to compute the irrationality exponent of some given explicit transcendental number.  It is known that \(\nu(e) = 2\) but according to \cite{Weisstein:IrrationalityMeasure}, for numbers like \(\pi\), \(\pi^2\), \(\log 2\) and \(\zeta(3)\) the current best known upper bounds lie somewhere between \(2.5\) and \(7.7\).  Even worse, it was established recently that there are computable numbers with uncomputable irrationality exponent \cite{Becher-et-al:IrrationalityExponents}.

In a moment we will have reason to be interested in the irrationality exponent of real numbers \(t \in (0,1)\) with the property that \(\xi = e^{\pi \ima t}\) is an algebraic number.  Of course this is only interesting if \(t\) is irrational so that \(\xi\) is not a root of unity.  It turns out that then \(t\) must be transcendental as a consequence of the following result.

\begin{theorem}[Gelfond--Schneider]
Let \(\alpha\) and \(\beta\) be algebraic numbers with \(\alpha \neq 0,1\) and \(\beta\) irrational.  Then each value of \(\alpha^\beta\) is a transcendental number.
\end{theorem}

The transcendence of \(t\) follows by setting \(\alpha = -1 = e^{\pi \ima}\) and \(\beta = t\).  So per se \(t\) could have any irrationality exponent \(\nu \ge 2\).  However, we can construct upper bounds on the irrationality exponent of \(t\) from the theory of linear forms in logarithms.  To explain this we observe that the Gelfond--Schneider theorem has the equivalent formulation that if (some values of) \(\log \alpha_1\) and \(\log \alpha_2\) are linearly independent over \(\Q\) for \(\alpha_1\) and \(\alpha_2\) nonzero algebraic, then \(\log \alpha_1\) and \(\log \alpha_2\) are also linearly independent over \(\overline{\Q}\).  For applications to Diophantine equations, however, it is not only important that the \emph{linear form in two logarithms}
\[ \beta_1 \log \alpha_1 + \beta_2 \log \alpha_2 \]
does not represent zero over \(\overline{\Q}\) nontrivially.  More so, explicit lower bounds of this form in terms of the degrees and heights of \(\beta_1\) and \(\beta_2\) are of interest.  For the problem we have in mind it suffices to consider the case when \(\beta_1\) and \(\beta_2\) are rational integers.

\begin{theorem}
  Let \(\alpha_1, \alpha_2 \neq 0, 1\) be algebraic numbers, let \(\beta_1, \beta_2\) be rational integers and let \(B = \max\{|\beta_1|, |\beta_2|\}\).  Then there is a constant \(C > 1\), which only depends on the heights and degrees of \(\alpha_1\) and \(\alpha_2\), such that
  \[|\beta_1 \log \alpha_1 + \beta_2 \log \alpha_2|  > B^{-C} \]
whenever \(\beta_1 \log \alpha_1 + \beta_2 \log \alpha_2\) is nonzero.
\end{theorem}

The generalization of this theorem from two to any number of logarithms and with any algebraic numbers as coefficients \(\beta_i\) is the content of \emph{Baker's theorem}.  So let us refer to the above result as the \emph{little Baker theorem} and let us call \(C = C_{\alpha_1, \alpha_2}\) a \emph{Baker constant} for \(\alpha_1\) and \(\alpha_2\) even though in the above form the result is already contained in earlier work of N.\,I.\,Feldman~\cite{Feldman:Improvement}.  Now since \(t = \frac{\log \xi}{\log (-1)}\) is the ratio of two logarithms as above, we conclude
\[ \left|t - \frac{m}{n}\right| = \frac{\left| n \,\log \xi - m\,\log(-1) \right|}{\pi n}  > \frac{1}{\pi n^{C+1}} \]
where we could assume \(0 \le m \le n\) because \(0 < t < 1\).  It follows that \(\nu(t) \le C+1\).  A lot of work has been dedicated in the literature to find explicit values of Baker constants.  To give the reader an impression of a typical result we extract from \cite{Baker-Wuestholz:LogarithmicForms}*{p.\,20} that in our setting
\[ C = 2^{40} d^8 \log H \]
is a possible choice for \(C\).  Here \(d\) is the degree of \(\xi\) and \(H\) is the \emph{height} of \(\xi\): the maximal absolute value of the relatively prime integer coefficients of the minimal polynomial of \(\xi\).  The factor \(2^{40}\) has order of magnitude \(10^{12}\) (one trillion).  Substantial improvements are possible by using more specifically that we are dealing with two logarithms only.  A result in this direction is due to N.\,Gouillon~\cite{Gouillon:ExplicitLowerBounds}*{Corollary~2.2 and below} which implies in our situation that for large \(n\) we have the bound \(|n \log \xi - m \log (-1)| > \textup{const} \ n^{-C}\) with
\[ C = 550 \pi \,d^2 \left(3.776 + 2.662 \,d+ 0.946 \,d \log d \right) \max \left\{ 2\pi, \log M(p_\xi) \right\} \]
where \(\log M(p_\xi)\) is the \emph{logarithmic Mahler measure} of the minimal polynomial \(p_\xi\) of \(\xi\),
\[ \log M(p_\xi) = \frac{1}{2\pi} \int_{0}^{2\pi} \log(|p_\xi(e^{i\theta})|)\, \diff \theta. \]
The value \(\frac{\log M(p_\xi)}{d}\) is also known as the \emph{absolute logarithmic height} of \(\xi\).  For a completely explicit example let us consider \(\xi = \frac{1}{4}(3 + \ima \sqrt{7})\) which has \(p_\xi(z) = 2z^2 -3z +2\) so that the above formula gives that
\[ C = C_{\xi, -1} = \num{452130} \]
is a Baker constant for \(\xi\) and \(-1\).  It should be possible to push things still further and get smaller constants in particular cases by going into the details of the main result in \cite{Gouillon:ExplicitLowerBounds} or by adapting a recent paper by Bugeaud~\cite{Bugeaud:Effective} from positive rational numbers to more general algebraic numbers.  But there does not seem to be much point in it because one cannot expect to get anywhere near the correct value of the irrationality exponent of \(t\) with this method.  So let us fix in a theorem what we should take home from this excursion.

\begin{theorem} \label{thm:excursionsummary}
Let \(t \in (0,1)\) be such that \(\xi = e^{\pi \ima t}\) is an algebraic number.  Then \(t\) is either rational or transcendental.  If \(t\) is rational, we have \(\nu(t) = 1\).  If \(t\) is transcendental, we have \(\nu(t) \ge 2\) but \(t\) is never a Liouville number: there are upper bounds for \(\nu(t)\) in terms of the degree and height of \(\xi\).
\end{theorem}

The results behind this theorem are by now classical but be aware that they are deep.  The Gelfond--Schneider theorem famously settled Hilbert's seventh problem while Baker's theorem---as well as the Thue--Siegel--Roth theorem---was worth a Fields medal.

\section{Proof of the main theorem}
\label{sec:proof}

Our main result Theorem~\ref{thm:maintheorem} explains how to read off shrinkage rates of knots from the reduced Alexander polynomial.  We remind the reader that \(\Lambda(1) = \pm 1 \neq 0\) and also \(\Lambda(-1) \neq 0\) because \(\Lambda\) divides \(\Delta\) and \(\Delta(-1)\) is the determinant of the knot which is known to be an odd integer.  So 
\[ \xi_1 = e^{\pi \ima t_1}, \,\ldots\,,\, \xi_u = e^{\pi \ima t_u} \quad \text{with} \quad t_1, \ldots, t_u \in (0,1) \]
 and their conjugates \(\overline{\xi}_1, \ldots, \overline{\xi}_u\) are all roots of \(\Lambda\) on the unit circle.  As a first step towards the proof we show the following sharpening of Proposition~\ref{prop:dependencyonreducedalex} which already shows that only the first reduced Alexander polynomial \(\Lambda\) is relevant for shrinkage rates.

\begin{proposition} \label{prop:evaluationinreducedalex}
  For \(n \gg 1\) let \(\widehat{\sigma}_n\) be the smallest positive value that the function \(z \mapsto |\Lambda(z)|\) attains at an \(n\)-th root of unity.  Then \((\widehat{\sigma}_n)_{n>1}\) and \((\sigma_n)_{n>1}\) are dilatationally equivalent.
\end{proposition}

\begin{proof}
  We retain the notation from the proof of Proposition~\ref{prop:dependencyonreducedalex}.  There we had already seen that \((\sigma_n)_{n>1}\) and \((\widetilde{\sigma}_n)_{n>1}\) are dilatationally equivalent where \(\widetilde{\sigma}_n\) is the smallest positive singular value of the matrix \((SAT)_n\).  Here \((SAT)_n \in M(r,s,\Z[z]/(z^n-1))\) is viewed as a matrix in \(M(nr, ns; \Z)\) by replacing the generator \(z\) with the cyclic \((n \times n)\)-matrix~\eqref{eq:circulantmatrix}.  The latter matrix is unitarily equivalent to the diagonal matrix whose entries are the \(n\)-th roots of unity.  It follows from the spectral mapping theorem that \(\widetilde{\sigma}_n\) is the smallest positive element of the set
  \[ \left\{ \left|\Lambda_i(e^{\frac{2 \pi \ima l}{n}})\right|\, \colon \ i=1, \ldots, k; \ l=0, \ldots, n-1 \right\}. \]
  Next we argue that we can disregard the higher reduced Alexander polynomials.  By definition \(\widehat{\sigma}_n\) denotes the smallest positive element of the subset
  \[ \left\{ \left|\Lambda(e^{\frac{2 \pi \ima l}{n}})\right|\, \colon \ l=0, \ldots, n-1 \right\} \]
of the above set where \(i = 1\) is fixed.  Since each \(\Lambda_i\) divides \(\Lambda\), we can consider the constant
  \[ D = \max_{1 \le i \le k} \sup_{|z|=1} \left| \frac{\Lambda(z)}{\Lambda_i(z)} \right| \ge 1.\]
  For given \(n\) let \(i = i(n)\) and \(l = l(n)\) be such that \(\widetilde{\sigma}_n = |\Lambda_i(e^{\frac{2 \pi \ima l}{n}})|\).  Since \(\Lambda\) has only finitely many roots, we must also have \(\Lambda(e^{\frac{2 \pi \ima l}{n}}) \neq 0\) whenever \(n\) is large enough.  Thus for large \(n\) we have
  \[ \widetilde{\sigma}_n = \left| \frac{\Lambda_i(e^{\frac{2\pi \ima l}{n}})}{\Lambda(e^{\frac{2\pi \ima l}{n}})} \right| \left| \Lambda(e^{\frac{2\pi \ima l}{n}})\right| \ge D^{-1} \left|\Lambda(e^{\frac{2\pi \ima l}{n}})\right| \ge D^{-1} \widehat{\sigma}_n. \]
  Thus we have \(\widetilde{\sigma}_n \le \widehat{\sigma}_n \le D \widetilde{\sigma}_n\) which completes the proof.
\end{proof}

The proposition says that small positive values of \(|\Lambda(z)|\) at \(n\)-th roots of unity determine the shrinkage rates.  At least for big \(n\) the function \(|\Lambda(z)|\) will take its smallest positive value at an \(n\)-th root of unity which is close to one of the zeros of \(\Lambda\).  For our purpose it is thus of interest to find estimates on the distances of \(n\)-th roots of unity from a given distribution of points on the unit circle.  A key result in this direction says that for any such distribution of points there are infinitely many regular \(n\)-gons each of whose vertices either coincides with one of the given points or is relatively far away from all of them.  This is made precise by the following lemma.

\begin{lemma} \label{lemma:farawayfromzeros}
  Let \(z_1, \ldots, z_u \in S^1\) be distinct points on the unit circle.  Then there is a constant \(0 < R < \frac{1}{2}\) and there are infinitely many positive integers \(n_j\) such that for each \(t = 1, \ldots, u\) and \(k = 1, \ldots, n_j\) either
  \[ z_t = e^{\frac{2 \pi \ima k}{n_j}} \quad \text{or} \quad \left|z_t - e^{\frac{2 \pi \ima k}{n_j}}\right| \ge 2 \sin \left(\frac{R \pi}{n_j}\right). \]
\end{lemma}

This lemma is a special case of \cite{Kammeyer:Approximating}*{Proposition~12} and arguing similarly as there we obtain the first half of our desired result.

\begin{proposition} \label{prop:upperboundonlowerrate}
    If \(\Lambda\) has roots on the unit circle, then \(\muinf(K) = \max\limits_{j=1,\ldots,u} \{\muinf_j\}\).
\end{proposition}

\begin{proof}
Without loss of generality we may assume \(\muinf_1 = \max_{j=1,\ldots,u} \{\muinf_j\}\).  The roots of \(\Lambda\) on \(S^1\) determine a constant \(0 < R < \frac{1}{2}\) as well as the sequence of positive numbers \((n_j)_j\) according to Lemma~\ref{lemma:farawayfromzeros}.  Choose \(0 < \delta < \frac{1}{2}\) so small that no point on the unit circle has more than one root of \(\Lambda\) in its \(\delta\)-ball.  Let \(\zeta_j\) be an \(n_j\)-th root of unity for which \(\widehat{\sigma}_{n_j} = |\Lambda(\zeta_j)|\).  Since \(|\Lambda(z)| = |\Lambda(\overline{z})|\), we can assume that \(\zeta_j\) lies on the upper half of the unit circle.  For large enough \(j\) precisely one root \(\xi_t = \xi_{t(j)}\) of \(\Lambda\) lies in the \(\delta\)-ball around \(\zeta_j\).  By Lemma~\ref{lemma:farawayfromzeros} we have \(|\xi_t - \zeta_j| \ge 2 \sin \left( \frac{R \pi}{n_j} \right)\) or even  \(|\xi_t - \zeta_j| \ge 2 \sin \left(\frac{\pi}{n_j}\right)\), which is the edge length of the regular \(n_j\)-gon, in case \(\xi_t\) happens to be an \(n_j\)-th root of unity.  Let \(d\) be the constant
  \begin{equation} \label{eq:lowerboundforrestofpolynomial} d = \inf_{0 \le t \le 1} \left|\frac{\Lambda(e^{\pi \ima t})}{\prod_{s=1}^u (e^{\pi \ima t}-\xi_s)^{\muinf_s}}\right| > 0. \end{equation}
  Then we obtain the estimate
  \begin{align*}
    \widehat{\sigma}_{n_j} = |\Lambda(\zeta_j)| & \ge d \left({\textstyle \prod\limits_{s \neq t}} \delta^{\muinf_s}\right) \left( 2 \sin \left(\frac{R \pi}{n_j} \right) \right)^{\muinf_t} \ge \\
    & \ge d \left({\textstyle \prod\limits_{s = 1}^u} \delta^{\muinf_s}\right) \left( 2 \sin \left(\frac{R \pi}{n_j} \right) \right)^{\muinf_1} \ge \frac{\textup{const}}{n_j^{\ \muinf_1}}
  \end{align*}
  where we used \(2 \sin(x) \ge x\) for small \(x > 0\).  Together with Proposition~\ref{prop:evaluationinreducedalex} this shows that \(\muinf(K) \le \muinf_1\).

  To prove the other inequality we consider the constant
  \begin{equation} \label{eq:upperboundforrestofpolynomial} D = \sup_{|z|=1} \left| \frac{\Lambda(z)}{(z-\xi_1)^{\muinf_1}} \right| > 0. \end{equation}
  There is an \(n\)-th root of unity \(\zeta_n\) closest but not equal to \(\xi_1\) and the distance \(|\zeta_n - \xi_1|\) is bounded from above by the edge length of the regular \(n\)-gon \(2 \sin \left(\frac{\pi}{n}\right)\).  Again we must have \(\Lambda(\zeta_n) \neq 0\) for big enough \(n\) so that
  \[ \widehat{\sigma}_n \le |\Lambda(\zeta_n)| \le D \left(2 \sin \left(\frac{\pi}{n}\right)\right)^{\muinf_1} \le \frac{\textup{const}}{n^{\muinf_1}}, \]
  where we used \(\sin(x) \le x\) for small \(x > 0\).  It follows that \(\muinf(K) \ge \muinf_1\).
\end{proof}

In the above proof we saw that regular \(n\)-gons with vertices far away from the roots of \(\Lambda\) were responsible for the precise value of the lower shrinkage rate.  Correspondingly, regular \(n\)-gons with vertices close to some given root or roots give rise to the upper shrinkage rate as we will see next.  Such \(n\)-gons occur whenever one of the numbers \(t_1, \ldots, t_u\) has an exceptionally good rational approximation.

\begin{proposition}
If \(\Lambda\) has roots on the unit circle, then \(\musup(K) = \max\limits_{j=1,\ldots,u} \{ \musup_j \}\).
\end{proposition}

\begin{proof}
  Without loss of generality we may assume \(\musup_1 = \nu_1 \muinf_1 = \max\limits_{j=1,\ldots,u} \{ \musup_j\}\).  Let \(\varepsilon > 0\).  By the definition of the irrationality exponent there exists a sequence of pairs of positive integers \((m_j, n_j)\) with \((n_j)_j\) monotonically increasing such that \(0 < \left|t_1 - \frac{m_j}{n_j}\right| < \frac{1}{n_j^{\nu_1 -\varepsilon}}\).  Let \(\zeta_j = e^{\frac{\pi \ima m_j}{n_j}}\) be the corresponding \(2 n_j\)-th root of unity.  For \(|z| \le 1\) we have \(|e^z - 1| \le \frac{7}{4}|z|\) from which we conclude
  \[ |\zeta_j - \xi_1| \le \left|e^{\frac{\pi \ima m_j}{n_j}} - e^{\pi \ima t_1} \right| = \left|1 - e^{\pi \ima\left(t_1 - \frac{m_j}{n_j}\right)}\right| \le \frac{7 \pi}{4} \left|t_1 - \frac{m_j}{n_j}\right| \le \frac{7 \pi}{4 n_j^{\nu_1 -\varepsilon}}. \]
  Let \(D\) be the constant from \eqref{eq:upperboundforrestofpolynomial} above.  For large enough \(j\) we must have \(\Lambda(\zeta_j) \neq 0\) hence
  \[ \widehat{\sigma}_{2 n_j} \le |\Lambda(\zeta_j)| \le D |\zeta_j - \xi_1|^{\muinf_1} \le \frac{\textup{const}}{n_j^{(\nu_1 - \varepsilon) \muinf_1}} = \frac{\textup{const}'}{(2 n_j)^{(\nu_1 - \varepsilon) \muinf_1}}. \]
  It follows that \(\musup(K) \ge (\nu_1 -\varepsilon) \muinf_1\) for all \(\varepsilon > 0\), thus \(\musup(K) \ge \musup_1\).

  On the other hand there is some constant \(N \gg 0\) such that for all \(\frac{m}{n} \in \Q\) with \(n \ge N\) and all \(j = 1, \ldots, u\) we have either \(t_j = \frac{m}{n}\) or \(\left|t_j - \frac{m}{n}\right| \ge \frac{1}{n^{\nu_j + \varepsilon}}\).  From \(|1 - e^z| \ge \frac{|z|}{4}\) for \(|z| \le 1\) it follows that for all \(2n\)-th roots of unity \(\zeta_n = e^{\frac{\pi \ima m}{n}}\) and all \(j = 1, \ldots, u\) either \(\zeta_n = \xi_j\) or
  \[ |\zeta_n - \xi_j| \ge \frac{\pi}{4} \left|t_j - \frac{m}{n}\right| \ge \frac{\textup{const}}{n^{\nu_j + \varepsilon}}. \]
Let \(d\) and \(\delta\) be as in the proof of Proposition~\ref{prop:upperboundonlowerrate} above.  For \(n \ge N\) let \(\zeta_n\) be an \(n\)-th root of unity on the upper half of the unit circle for which \(\widehat{\sigma}_n = |\Lambda(\zeta_n)|\).  Again let \(\xi_t = \xi_{t(n)}\) be the only root of \(\Lambda\) within the \(\delta\)-ball around \(\zeta_n\), assuming \(n\) is sufficiently large.  Since an \(n\)-th root of unity is also a \(2n\)-th root of unity,  we can estimate with the above that
  \[ \widehat{\sigma}_n = |\Lambda(\zeta_n)| \ge d \left({\textstyle \prod\limits_{s \neq t}} \delta^{\muinf_s}\right) |\zeta_n - \xi_t|^{\muinf_t} \ge \frac{\textup{const}}{n^{(\nu_t + {\varepsilon})\muinf_t}} \ge \frac{\textup{const}}{n^{\musup_1 + \varepsilon \max_s \{ \muinf_s \}}}. \]
  It follows that \(\musup(K) \le \musup_1 + \varepsilon \max_s \{ \muinf_s \}\) for all \(\varepsilon > 0\), thus \(\musup(K) \le \musup_1\).
\end{proof}

\begin{proof}[Proof of Theorem~\ref{thm:maintheorem}]
  By the above two propositions only the case when \(\Lambda\) has no roots on the unit circle still needs attention.  Assuming this, the constant \(\inf_{|z|=1} |\Lambda(z)|\) is a positive lower bound for the sequence \((\widehat{\sigma}_n)_{n>1}\).  For every CW structure of \(X\) the sequence \((\sigma_n)\) is dilatationally equivalent to \((\widehat{\sigma}_n)\) by Proposition~\ref{prop:evaluationinreducedalex} so \((\sigma_n)\) is likewise positively bounded from below.  This means in particular that \((\sigma_n)\) has vanishing shrinkage rate so \(\mu(K) = 0\).
\end{proof}

\section{Shrinkage types of the zoo of knots}
\label{sec:zoo}

In this section we investigate the taxonomy shrinkage types and shrinkage rates impose on the zoo of knots.  By our main result, Theorem~\ref{thm:maintheorem}, the type~I knots are the knots for which a fixed spectral gap remains for all finite cyclic coverings.  The type~II knots are the ones where the spectral gap closes with an eventually constant rate whereas for the type~III knots the gap closes with an oscillating rate.  Theorem~\ref{thm:twistandtorusknots} from the introduction says that all three cases occur within the two classical infinite families of prime knots as we will now proof.

\begin{proof}[Proof of Theorem~\ref{thm:twistandtorusknots}.]
  The normalized reduced Alexander polynomial of the twist knot with \(2n\) half-twists is \(\Lambda(z) = nz^2 -(2n+1)z + n\).  This includes the case \(n=0\) of the unknot when \(\Lambda\) is a unit in \(\Z[z^{\pm 1}]\).  So the unknot has shrinkage type~I.  For \(n > 0\) the polynomial is quadratic with discriminant \(4n+1 > 0\).  Thus it has two distinct real roots which means the knot has shrinkage type~I as well.

  The \((p,q)\)-torus knot \(T_{p,q}\) with \(p\) and \(q\) coprime has normalized reduced Alexander polynomial
  \[ \Lambda(z) = \frac{(z^{pq}-1)(z-1)}{(z^p-1)(z^q-1)} \]
  or, in other words, \(\Lambda(z) = \prod_\xi (z-\xi)\) where \(\xi\) runs over all \(pq\)-th roots of unity which are neither \(p\)-th nor \(q\)-th roots of unity.  Since each \(\xi\) has multiplicity and total multiplicity one, we have \(\muinf(T_{p,q}) = \musup(T_{p,q}) = 1\) whenever \(p,q \neq \pm 1\).   Thus every nontrival \(T_{p,q}\) has shrinkage type~II .

  The reduced normalized Alexander polynomial of the twist knot \(K\) with \(2n-1\) half-twists is given by \(\Lambda(z) = nz^2 -(2n-1)z +n\).  For \(n \ge 2\) this polynomial is irreducible over \(\Z\) but not monic.  Thus it is not cyclotomic and its zeros are never roots of unity.  Nevertheless, the root
  \[ \xi = \frac{1}{2n}\left(2n-1 + \ima \sqrt{4n-1}\right) \]
  of \(\Lambda(z)\) lies on the unit circle.  So while the multiplicity of \(\xi\) is one, Theorem~\ref{thm:excursionsummary} says that the total multiplicity is at least two.  It follows that \(\muinf(K) = 1\) while \(\musup(K) \ge 2\), thus \(K\) has shrinkage type~III.
\end{proof}

Note that the proof actually shows torus knots are of type~\(\textup{II}_1\) while odd twist knots are of type \(\textup{III}_1\).  As discussed in Section~\ref{sec:excursion}, we cannot say much about upper shrinkage rates of type~\(\textup{III}\) knots except for the bounds coming from linear forms in logarithms.  For example the three-twist knot \(5_2\) has \(\Lambda(z) = 2z^2 -3z +2\) and for this polynomial we computed a Baker constant of \num{452130}.  So the three-twist knot is of type \(\textup{III}_1^{\le \num{452131}}\).  Next we prove Theorem~\ref{thm:knotsofalltypesexist} which asserts that knots of type~\(\textup{II}_n\) and \(\textup{III}_n\) exist for any \(n\).

\begin{proof}[Proof of Theorem~\ref{thm:knotsofalltypesexist}.]
  Let  \(p(z) = z^2-z+1\).  By Theorem~\ref{thm:levine} we can realize the polynomials \(\Lambda_1(z) = (p(z))^n\) and \(\Lambda_2(z) = 1\) as reduced Alexander polynomials of a knot~\(K\).  By Theorem~\ref{thm:maintheorem} the knot~\(K\) is of type~\(\textup{II}_n\).  Let \(K = K_1 \# \cdots \# K_r\) be the prime decomposition of~\(K\).  Accordingly, we have an isomorphism of \(\Q[z^{\pm 1}]\)-modules
  \[ H_1(X_\infty^K;\Q) \cong H_1(X_\infty^{K_1};\Q) \oplus \cdots \oplus H_1(X_\infty^{K_r};\Q) \]
for the first rational homology of the infinite cyclic coverings.  By construction \(H_1(X_\infty^K; \Q) \cong \Q[z^{\pm 1}] / (p^n(z))\) is cyclic.  Moreover, \((p^n(z))\) is a primary ideal in the principal ideal domain \(\Q[z^{\pm 1}]\) because \(p(z)\) is irreducible.  Thus all but one of the summands \(H_1(X_\infty^{K_i})\) is trivial, hence all but one knot \(K_{i_0}\) have trivial reduced (and nonreduced) Alexander polynomials.  It follows that \(K_{i_0}\) is a prime knot with \(\Lambda_1(z) = (p(z))^n\) and thus has type~\(\textup{II}_n\).  Similarly we obtain a type~\(\textup{III}_n\) prime knot using \(q(z) = 2z^2-3z+2\) instead of \(p(z)\).
\end{proof}

We list the shrinkage types of prime knots with up to eight crossings in Table~\ref{tab:knottypes}.  The knots are enumerated by Alexander--Briggs notation as continued by Rolfsen~\cite{Rolfsen:KnotsAndLinks} and the normalized reduced Alexander polynomials are included as a product of irreducible factors.  To determine the shrinkage type of the knots we had to find the roots of these polynomials on the unit circle.  Since \(\Lambda\) is of even degree and has palindromic coefficients, there is an elementary way to do this which is nicely explained in~\cite{Conrad:NumbersOnCircle}*{Theorem~2.3}.  The three non-alternating knots in the list are of type~II.  To prevent the reader from uttering any astronomic conjectures from this small set of data we should say that \(9_{42}\) and \(9_{44}\) are non-alternating knots of type~\(\textup{III}_1\) and I, respectively.  Amongst the prime knots up to ten crossings there is no knot of which we can say for sure it is of type~\(\textup{III}_2\); there are only the three possible candidates \(10_{65}\), \(10_{77}\) and \(10_{82}\) mentioned in the introduction.  However, as exploited in the proof of Theorem~\ref{thm:knotsofalltypesexist}, any knot with reduced Alexander polynomial \(\Lambda(z) = (2z^2-3z+2)^2\) is definitely of type~\(\textup{III}_2\).  With the help of ``KnotInfo''~\cite{Cha-Livingston:KnotInfo} we checked that among the 2977 prime knots up to twelve crossings there are precisely five with this reduced Alexander polynomial.  All of them have twelve crossings and only one of them is alternating: the knot \(12\textup{a}_{169}\) as pictured in Figure~\ref{fig:12a169}.
\begin{figure}[htb]
  \includegraphics[width=3.8cm]{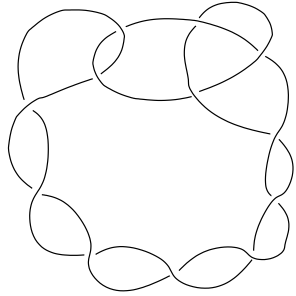}
  \caption{The knot \(12\textup{a}_{169}\), a prime knot of type \(\textup{III}_2\).}
  \label{fig:12a169}
\end{figure}

In the introduction we explained that Novikov--Shubin numbers originally motivated our study of shrinkage rates.  So we should briefly confirm that the lower shrinkage rate equals the reciprocal of the second Novikov--Shubin number \(\alpha^{(2)}_2(X_\infty)\) of the infinite cyclic covering \(X_\infty\).  For the definition of Novikov--Shubin numbers, consult~\cite{Lueck:L2Invariants}*{Chapter~2}.

\afterpage{
  \setlength{\tabcolsep}{6pt}
  \renewcommand{\arraystretch}{1.5}
  \begin{longtable}{|clp{4cm}cp{2.6cm}|}
    \hline
    \endfoot
    \caption{Shrinkage types of prime knots up to eight crossings.}
    \label{tab:knottypes} \\
    \hline
    \textbf{knot} & \textbf{name} & \(\bm{\Lambda(z)}\) & \textbf{type} & \textbf{remark} \\
    \hline
    \endfirsthead
    \hline
    \textbf{knot} & \textbf{name} & \(\bm{\Lambda(z)}\) & \textbf{type} & \textbf{remark} \\
    \hline
    \endhead
  \(3_1\) & trefoil & \(z^2-z+1\) & \(\textup{II}_1\) & torus knot \\
  \(4_1\) & figure eight & \(z^2-3z+1\) & I & twist knot\\
  \(5_1\) & cinquefoil & \(z^4-z^3+z^2-z+1\) & \(\textup{II}_1\) & torus knot \\
  \(5_2\) & three-twist & \(2z^2-3z+2\) & \(\textup{III}_1\) & twist knot \\
  \(6_1\) & Stevedore & \((2z-1)(z-2)\) & I & twist knot \\
  \(6_2\) & Miller inst. & \(z^4-3z^3+3z^2-3z^2+1\) & \(\textup{III}_1\) & \\
  \(6_3\) & & \(z^4-3z^3+5z^2-3z+1\) & I & \\
  \(7_1\) & septoil & \(z^6-z^5+z^4-z^3+z^2-z+1\) & \(\textup{II}_1\) & torus knot \\
  \(7_2\) & five-twist & \(3z^2-5z+3\) & \(\textup{III}_1\) & twist knot \\
  \(7_3\) & & \(2z^4-3z^3+3z^2-3z+2\) & \(\textup{III}_1\) & \\
  \(7_4\) & endless & \(4z^2-7z+4\) & \(\textup{III}_1\) & \\
  \(7_5\) & & \(2z^4-4z^3+5z^2-4z+2\) & \(\textup{III}_1\) & \\
  \(7_6\) & & \(z^4-5z^3+7z^2-5z+1\) & I & \\
  \(7_7\) & & \(z^4-5z^3+9z^2-5z+1\) & I & \\
  \(8_1\) & six-twist & \(3z^2-7z+3\) & I & twist knot \\
  \(8_2\) & & \(z^6-3z^5+3z^4-3z^3+\) \newline \(+3z^2-3z+1\) & \(\textup{III}_1\) & \\
  \(8_3\) & & \(4z^2-9z+4\) & I & \\
  \(8_4\) & & \(2z^4-5z^3+5z^2-5z+2\) & \(\textup{III}_1\) & \\
  \(8_5\) & & \((z^2-z+1)\cdot\) \newline \(\cdot (z^4-2z^3+z^2-2z+1)\) & \(\textup{III}_1\) & \(\Lambda\) has two kinds of zeros on \(S^1\) \\
  \(8_6\) & & \(2z^4-6z^3+7z^2-6z+2\) & \(\textup{III}_1\) & \\
  \(8_7\) & & \(z^6-3z^5+5z^4-5z^3+\) \newline \(+5z^2-3z+1\) & \(\textup{III}_1\) & \\
  \(8_8\) & & \(2z^4-6z^3+9z^2-6z+2\) & I & \\
  \(8_9\) & & \(z^6-3z^5+5z^4-7z^3+\) \newline \(+5z^2-3z+1\) & I & \\
  \(8_{10}\) & & \((z^2-z+1)^3\) & \(\textup{II}_3\) & first prime knot with \(\muinf(K) \neq 1\) \\ 
  \(8_{11}\) & & \((2z-1)(z-2)(z^2-z+1)\) & \(\textup{II}_1\) & non-torus knot of type \(\textup{II}_1\) \\
  \(8_{12}\) & & \(z^4-7z^3+13z^2-7z+1\) & I & \\
  \(8_{13}\) & & \(2z^4-7z^3+11z^2-7z+2\) & I & \\
  \(8_{14}\) & & \(2z^4-8z^3+11z^2-8z+2\) & \(\textup{III}_1\) & \\
  \(8_{15}\) & & \((z^2-z+1)(3z^2-5z+3)\) & \(\textup{III}_1\) & \(\Lambda\) has two kinds of zeros on \(S^1\) \\
  \(8_{16}\) & & \(z^6-4z^5+8z^4-9z^3+\) \newline \(+8z^2-4z+1\) & \(\textup{III}_1\) & \\
  \(8_{17}\) & & \(z^6-4z^5+8z^4-11z^3+\) \newline \(+8z^2-4z+1\) & I & \\
  \(8_{18}\) & carrick mat & \((z^2-z+1)(z^2-3z+1)\) & \(\textup{II}_1\) & \(\Lambda(z) \neq \Delta(z)\) \\
  \(8_{19}\) & & \((z^2-z+1)(z^4-z^2+1)\) & \(\textup{II}_1\) & non-alternating torus knot \\
  \(8_{20}\) & & \((z^2-z+1)^2\) & \(\textup{II}_2\) & non-alternating \\
  \(8_{21}\) & & \((z^2-z+1)(z^2-3z+1)\) & \(\textup{II}_1\) & non-alternating \\
  \end{longtable}
}

\begin{theorem} \label{thm:relationtons}
A knot \(K\) has shrinkage type~\textup{I} if and only if \(\alpha^{(2)}_2(X_\infty) = \infty^+\).  For all knots of type~\textup{II} and~\textup{III} we have \(\alpha^{(2)}_2(X_\infty) = \muinf(K)^{-1}\).
\end{theorem}

\begin{proof}
  Compare \cite{Kammeyer:Approximating}*{p.\,11}.  We obtain the \(\ell^2\)-chain complex \(C^{(2)}_*(X_\infty)\) of \(X_\infty\) by applying the functor \(\ell^2 \Z \otimes_{\Z[\Z]}\) to the cellular chain complex of \(X_\infty\).  The second \(\ell^2\)-differential
  \[ C^{(2)}_2(X_\infty) \overset{d^{(2)}_2}{\longrightarrow} C^{(2)}_1(X_\infty) \]
  is still of the form \((\ell^2 \Z)^{n-1} \stackrel{\cdot A}{\longrightarrow} (\ell^2 \Z)^n\) so that it is given by right multiplication with \(A \in M(r,s;\Z[z^{\pm 1}])\).  Again we transform \(A\) into Smith normal form \(SAT\).  From \cite{Lueck:L2Invariants}*{Lemma~2.11\,(9), p.\,77; Lemma~2.15\,(1), p.\,80} we obtain
  \[ \alpha^{(2)}_2(X_\infty) = \alpha^{(2)}(A) = \alpha^{(2)}(SAT) = \min_{i=1, \ldots, k} \{ \alpha^{(2)}(\Lambda_i) \}. \]
  According to \cite{Lueck:L2Invariants}*{Lemma\,2.58, p.\,100} and its proof, the Novikov--Shubin number of a polynomial multiplication operator \(\ell^2 \Z \stackrel{\cdot p(z)}{\longrightarrow} \ell^2 \Z\) is given by \(\frac{1}{\nu}\) where \(\nu\) is the highest multiplicity of a root of \(p(z)\) on the unit circle.  If \(p(z)\) has no such root, the Novikov--Shubin number \(\alpha^{(2)}(p(z))\) is \(\infty^+\).  The relation \(\Lambda_{i+1} \mid \Lambda_i\) implies that multiplicities of roots of \(\Lambda_i\) are growing for decreasing \(i\).  Thus \(\min_i \{ \alpha^{(2)}(\Lambda_i) \} = \alpha^{(2)}(\Lambda_1)\) which completes the proof.
\end{proof}

We want to conclude this article by discussing some open questions.  Of course all questions involving upper shrinkage rates are directly related to calculating irrationality exponents of ratios of two logarithms and therefore answers do not seem in reach.  Nevertheless, let us ask: 

\begin{itemize}
\item \emph{Is the upper shrinkage rate always an integer?}
\item \emph{Are there examples of knots with some \(\nu_j > 2\)?}
\item \emph{What are the asymptotic percentages of shrinkage types for growing crossing number?}
\end{itemize}

On the conceptual side it would of course be desirable that shrinkage types had some additional geometric relevance beside describing the spectral gaps for cyclic coverings.

\begin{itemize}
\item \emph{Are there relations between shrinkage rates and other geometric concepts of knot theory?}
\end{itemize}

As negative answers we should mention that there are fibered and non-fibered as well as hyperbolic and non-hyperbolic knots of type~I, II and~III.  Note however, that a slice knot with trivial second Alexander polynomial has even lower shrinkage rate as a consequence of the Fox--Milnor condition.

\end{document}